\documentclass[11pt,a4paper]{amsart}

\usepackage[T1]{fontenc}
\usepackage{amssymb}
\usepackage{microtype}
\usepackage{fancyvrb}
\usepackage{xcolor}
\usepackage[left=28mm,right=28mm,top=28mm,bottom=30mm,
            headsep=8mm,footskip=12mm]{geometry}
\usepackage[colorlinks=true,linkcolor=black,citecolor=black,urlcolor=black]{hyperref}

\makeatletter
\@namedef{subjclassname@2020}{%
  \textup{2020} Mathematics Subject Classification}
\makeatother

\newcommand{\R}{\mathbb R}
\newcommand{\Z}{\mathbb Z}
\newcommand{\vol}{\operatorname{vol}}
\newcommand{\vect}[1]{\mathbf{#1}}
-lmtt10 at 7pt
\font\appendixcodenumberfont=ec-lmtt10 at 6pt

\newtheorem{theorem}{Theorem}[section]
\newtheorem{proposition}[theorem]{Proposition}
\newtheorem{lemma}[theorem]{Lemma}

\theoremstyle{definition}
\newtheorem{definition}[theorem]{Definition}
\newtheorem{remark}[theorem]{Remark}
\numberwithin{equation}{section}

\begin{document}

\title[Orthogonality defects of HKZ-reduced bases]
{Orthogonality Defects of HKZ-Reduced Bases}

\author[Z. Li]{Zhixing Li}
\address{School of Mathematics\\ Nanjing University\\
Nanjing 210093, P. R. China}
\email{zhixing.li.math@gmail.com}

\date{}

\begin{abstract}
Let \(D_n\) denote the supremum of the orthogonality defects over
all \(n\)-dimensional Hermite--Korkine--Zolotarev (HKZ)-reduced bases.
This paper shows that, as \(n\to\infty\),
\[
  \log D_n=2n\log n-n\log\log n+O(n).
\]
We also determine the exact value \(D_4=4375/1024\).
\end{abstract}

\subjclass[2020]{Primary 11H50; Secondary 11H55}

\keywords{lattice, orthogonality defect,
Hermite--Korkine--Zolotarev reduced basis}

\maketitle

\section{Introduction}
\label{sec:introduction}

Let \(\vect{b}_1,\ldots,\vect{b}_n\) be linearly independent vectors in
\(\R^n\).  The
discrete subgroup
\[
 \Lambda=\Lambda(B)
  =
  \left\{z_1\vect{b}_1+\cdots+z_n\vect{b}_n:
  z_1,\ldots,z_n\in\Z\right\}
\]
is an \(n\)-dimensional lattice, and
\(B=(\vect{b}_1,\ldots,\vect{b}_n)\) is an ordered
basis of \(\Lambda\).
For \(n\geq2\), an \(n\)-dimensional lattice has infinitely many bases, and any two of them
are related by a unimodular matrix.

The concept of a lattice, which lies at the foundation of the geometry of
numbers, was formally introduced by Minkowski \cite{minkowski1892}.  Its
origins can be traced back to Gauss and Hermite; see \cite{Gruber}.  From 1997
to 1998, three lattice-based cryptosystems, AD \cite{ajtaidwork1997}, GGH, and
NTRU \cite{hoffsteinpiphersilverman1998}, were invented.
In recent years, lattice-based cryptography has become a key player in
post-quantum cryptography; see \cite{Zong2025}.  The GGH cryptosystem,
introduced by O. Goldreich, S. Goldwasser, and S. Halevi \cite{ggh1997},
assumes that the lattices employed have relatively orthogonal bases.  In fact,
finding the most orthogonal basis of a given lattice is a basic problem in
lattice-based cryptography; see \cite{Zong2026}.  It is therefore natural to
ask: \emph{How orthogonal can the best basis of an \(n\)-dimensional lattice
be?}  This problem has already been studied by Hermite \cite{hermite1850},
Minkowski \cite{minkowski1905}, H. F. Blichfeldt \cite{blichfeldt1914},
L. Bieberbach and I. Schur \cite{bieberbachschur1928}, L. J. Mordell
\cite{mordell1944}, K. Mahler \cite{mahler1938,mahler1946},
B. L. van der Waerden \cite{vanderwaerden1956,vanderwaerden1969}, and many
others.

Let \(\det(\Lambda)\) denote the \emph{determinant} of the lattice, which is
the volume of the fundamental parallelotope \(P\) generated by the basis
\(B=(\vect{b}_1,\ldots,\vect{b}_n)\). \(\det(\Lambda)\) does not depend on the
chosen basis but only on the lattice itself. We define
\begin{equation}
  \Delta(B)
  =
  \frac{\prod_{i=1}^n\lVert \vect{b}_i\rVert^2}
       {\det(\Lambda)^2}=\frac{\prod_{i=1}^n\lVert \vect{b}_i\rVert^2}
       {\vol(P)^2}
  \label{eq:intro-defect}
\end{equation}
and
\[
  \Delta_n=\sup_{\dim\Lambda=n}\;
  \min_{B\text{ a basis of }\Lambda}\Delta(B).
\]
Hadamard's inequality gives \(\Delta(B)\geq1\), with equality if and only if
\(B\) is orthogonal.  Thus, in a certain sense, \(\Delta(B)\) measures the
orthogonality of the basis.  We will call \(\Delta(B)\) the
\emph{orthogonality defect} of the basis \(B\).  Therefore, the natural problem
mentioned above is to determine or estimate the values of \(\Delta_n\).

Lattices and positive definite quadratic forms are equivalent; see
\cite[p.~23]{Zong1999}.  In 1850, Hermite \cite{hermite1850} (see also
\cite[p.~39]{martinet2003}, \cite[Theorem~1.1]{martinet2014}) showed, in terms
of positive definite quadratic forms, that
\[
  \Delta_n \leq
  \left(\frac43\right)^{n(n-1)/2}.
\]
In other words, every \(n\)-dimensional lattice has a basis with
orthogonality defect not more than
\(\left(4/3\right)^{n(n-1)/2}\).  Since then, \(\Delta_n\) has been
studied by applying various reduction theories, including Minkowski reduction
\cite{minkowski1905,blichfeldt1914,bieberbachschur1928,remak1938,
mahler1938,mahler1946,vanderwaerden1956,vanderwaerden1969}, LLL reduction
\cite{lll1982}, and HKZ reduction, which is the focus of this paper.

Building on Hermite's work, Korkine and Zolotarev
\cite{korkinzolotarev1873} introduced the recursive reduction now called
Hermite--Korkine--Zolotarev (HKZ) reduction (the formal definition is given in
Section~2) and proved that every lattice has an HKZ-reduced basis.  We define
\begin{equation}
  D_n
  =
  \sup\bigl\{\Delta(B):
  B\text{ is an \(n\)-dimensional HKZ-reduced basis}\bigr\}.
  \label{eq:intro-Dn}
\end{equation}
Since every lattice admits an HKZ-reduced basis, we have
\[
  \Delta_n  \leq D_n.
\]

Let \(\lambda_1(\Lambda)\) denote the length of the shortest nonzero vector
of \(\Lambda\).  The \(n\)-dimensional Hermite constant is
\[
  \gamma_n
  =
  \sup_{\dim\Lambda=n}
  \frac{\lambda_1(\Lambda)^2}{\det(\Lambda)^{2/n}}.
\]
Since every basis vector has length at least \(\lambda_1(\Lambda)\), one
can easily deduce that
\begin{equation}
  D_n\geq\gamma_n^n.
\label{eq:intro-gamma-lower-bound}
\end{equation}
H.~Bennett \cite[Eq.~(13) and Proposition~7.5]{bennett2018} also proved
\begin{equation}
  D_n\geq4^{-n}n!.
\label{eq:intro-bennett-bound}
\end{equation}
Consequently, using the classical estimate \(
  \gamma_n \geq n/(2\pi e)(1+o(1))
\)
and Stirling's formula, both
\eqref{eq:intro-gamma-lower-bound} and
\eqref{eq:intro-bennett-bound} imply
\[
  \log D_n\geq n\log n-O(n).
\]

In the opposite direction, in 1990 J.~C. Lagarias, H.~W. Lenstra, Jr., and C. Schnorr
\cite[Theorem~2.3]{lagariaslenstraschnorr1990} proved that
\begin{equation}
  D_n
  \leq
  \gamma_n^n\prod_{i=1}^n\frac{i+3}{4}
  =
  \gamma_n^n\frac{(n+3)!}{6\cdot4^n}.
  \label{eq:intro-lls-bound}
\end{equation}
In recent years, J. Wen
and X. Chang \cite[Theorem~4]{wenchang2019} obtained
\begin{equation}
  D_n
  \leq
  \left(\frac n8+\frac65\right)^n
  \frac{(n+3)!}{6\cdot4^n},
  \label{eq:intro-wen-chang-bound}
\end{equation}
and C. Porter, E. Dable-Heath, and C. Ling \cite[Theorem~3]{porter2024} obtained
\begin{equation}
  D_n
  \leq
  \frac{25}{12}\gamma_{n-3}^{n-3}
  \prod_{i=4}^n\left(\frac i4+\frac{29}{24}\right)
  \qquad(n\geq4).
  \label{eq:intro-porter-bound}
\end{equation}
In fact, all \eqref{eq:intro-lls-bound}--\eqref{eq:intro-porter-bound} imply
\[
  \log D_n\leq2n\log n+O(n). 
\]

This paper gives the following asymptotic estimate of $D_n$.

\begin{theorem}
\label{thm:asymptotic-growth}
As \(n\to\infty\), one has
\begin{equation}
  \log D_n=2n\log n-n\log\log n+O(n).
  \label{eq:asymptotic-main}
\end{equation}
\end{theorem}

Porter, Dable-Heath, and Ling \cite[Theorem~2 and Section~4]{porter2024}
determined
\[
  D_1=1,\qquad D_2=\frac43,\qquad D_3=\frac{25}{12},
\]
and conjectured that \(D_4=\gamma_4^4=4\).  This paper disproves the
conjecture by proving the following theorem.

\begin{theorem}
\label{thm:rank-four-exact}
\begin{equation}
  D_4=\frac{4375}{1024}.
  \label{eq:rank-four-exact-value}
\end{equation}
Moreover, up to similarity, there is a unique lattice admitting an HKZ-reduced basis attaining \(D_4\).
\end{theorem}

\section{Lattices, HKZ Reduction, and Orthogonality Defect}
\label{sec:preliminaries}

Throughout the paper, all lattices are full-rank Euclidean lattices.  We
write \(\R\) and \(\Z\) for the real numbers and the integers,
respectively; \(\lVert\cdot\rVert\) and \(\langle\cdot,\cdot\rangle\)
denote the Euclidean norm and inner product.

Let \(B=(\vect{b}_1,\ldots,\vect{b}_n)\) be a lattice basis as in
Section~\ref{sec:introduction}.  For \(1\leq i\leq n\), let \(\pi_i\) be
the orthogonal projection onto
\(\operatorname{span}_{\R}(\vect{b}_1,\ldots,\vect{b}_{i-1})^\perp\), with
\(\pi_1\) the identity.  Writing \(\langle\cdot\rangle_{\Z}\) for integer
span, define the \(i\)th projected lattice by
\[
  \Lambda_i(B)
  =
  \pi_i(\Lambda(B))
  =
  \bigl\langle
    \pi_i(\vect{b}_i),\ldots,\pi_i(\vect{b}_n)
  \bigr\rangle_{\Z}.
\]
The projected vectors
\(\pi_i(\vect{b}_i),\ldots,\pi_i(\vect{b}_n)\) are linearly independent;
hence \(\Lambda_i(B)\) is a lattice.  Equivalently, it is the projection of
\(\Lambda(B)\) at level \(i\).

Write
\(
  \vect{b}_i^*=\pi_i(\vect{b}_i).
\)
Then the Gram--Schmidt decomposition of \(B\) is
\[
  \vect{b}_i
  =\vect{b}_i^*+\sum_{j<i}\mu_{i,j}\vect{b}_j^*,
  \qquad
  \mu_{i,j}
  =\frac{\langle\vect{b}_i,\vect{b}_j^*\rangle}
        {\langle\vect{b}_j^*,\vect{b}_j^*\rangle},
  \qquad
  d_i=\lVert\vect{b}_i^*\rVert^2.
\]

For a lattice \(\Lambda\), recall that its first minimum is
\[
  \lambda_1(\Lambda)
  =
  \min\{\lVert \vect{x}\rVert:
  \vect{0}\neq\vect{x}\in\Lambda\}.
\]

The squared norm of a lattice vector
\(\sum_{i=1}^n z_i\vect{b}_i\ (z_i\in\Z)\) has the Lagrange
expansion
\begin{equation}
  Q(z_1,\ldots,z_n)
  =
  \left\lVert\sum_{i=1}^nz_i\vect{b}_i\right\rVert^2
  =
  \sum_{j=1}^nd_j
  \left(z_j+\sum_{k>j}\mu_{k,j}z_k\right)^2.
  \label{eq:prelim-lagrange-expansion}
\end{equation}
We write \(Q_i\) for the norm form of \(\Lambda_i(B)\) in the basis
\((\pi_i(\vect{b}_i),\ldots,\pi_i(\vect{b}_n))\):
\[
  Q_i(z_i,\ldots,z_n)
  =
  \sum_{j=i}^nd_j
  \left(z_j+\sum_{k>j}\mu_{k,j}z_k\right)^2.
\]
We now give the precise Gram--Schmidt formulation of HKZ reduction.

\begin{definition}
\label{def:hkz}
An ordered basis \(B=(\vect{b}_1,\ldots,\vect{b}_n)\) is
\emph{Hermite--Korkine--Zolotarev reduced}, or \emph{HKZ-reduced}, if
\begin{enumerate}
  \item it is size-reduced:
  \[
    |\mu_{i,j}|\leq\frac12
    \qquad(1\leq j<i\leq n);
  \]
  \item at every level \(1\leq i\leq n\), 
  \(\vect{b}_i^*\) is a shortest nonzero vector of \(\Lambda_i(B)\), i.e.
  \[
    \lVert \vect{b}_i^*\rVert=\lambda_1(\Lambda_i(B)).
  \]
\end{enumerate}
\end{definition}

In terms of \eqref{eq:prelim-lagrange-expansion}, the second condition
is equivalent to
\[
  \min_{\vect{z}\in\Z^{n-i+1}\setminus\{\vect{0}\}}
  Q_i(\vect{z})=d_i
  \qquad(1\leq i\leq n).
\]
Equivalently, Definition~\ref{def:hkz} says that \(B\) is HKZ-reduced if
and only if it is size-reduced, \(\vect{b}_1\) is a shortest nonzero vector
of \(\Lambda(B)\), and the basis
\((\pi_2(\vect{b}_2),\ldots,\pi_2(\vect{b}_n))\) of \(\Lambda_2(B)\) is
HKZ-reduced.

\begin{proposition}
\label{prop:hkz-existence}
Every Euclidean lattice admits an HKZ-reduced basis.
\end{proposition}

\begin{proof}
We argue by induction on the dimension \(n\).  The one-dimensional case is
trivial.  Given an \(n\)-dimensional lattice \(\Lambda\), we construct an
HKZ-reduced basis \(B\) as follows.  First choose \(\vect{b}_1\in\Lambda\) with
\(\lVert\vect{b}_1\rVert=\lambda_1(\Lambda)\).
This shortest vector is primitive, meaning that it is not a proper integer multiple of another lattice vector, and therefore extends to a basis of \(\Lambda\). Hence \(\Lambda\cap\R\vect{b}_1=\Z\vect{b}_1\), and \(\pi_2(\Lambda)\) is an \((n-1)\)-dimensional lattice.
By induction, let
\((\vect{c}_2,\ldots,\vect{c}_n)\) be an HKZ-reduced basis of
\(\pi_2(\Lambda)\).  For each \(i=2,\ldots,n\), choose \(\alpha_i\in\R\) such
that \(\vect{c}_i+\alpha_i\vect{b}_1\in\Lambda\), and choose \(m_i\in\Z\)
such that \(\lvert\alpha_i-m_i\rvert\leq1/2\).  Put
\[
  \vect{b}_i
  =\vect{c}_i+(\alpha_i-m_i)\vect{b}_1
  =(\vect{c}_i+\alpha_i\vect{b}_1)-m_i\vect{b}_1\in\Lambda.
\]
The projections of these lifts form a basis of \(\pi_2(\Lambda)\), and \(\Lambda\cap\R\vect{b}_1=\Z\vect{b}_1\); hence \(B=(\vect{b}_1,\ldots,\vect{b}_n)\) is a basis of \(\Lambda\).
Then it is easy to check that
\((\pi_2(\vect{b}_2),\ldots,\pi_2(\vect{b}_n))\) is the HKZ-reduced basis
\((\vect{c}_2,\ldots,\vect{c}_n)\) of \(\pi_2(\Lambda)\), and that
\(B=(\vect{b}_1,\ldots,\vect{b}_n)\) is size-reduced.  Thus \(B\) is an
HKZ-reduced basis of \(\Lambda\).
\end{proof}

Let \(G=(\langle \vect{b}_i,\vect{b}_j\rangle)_{i,j}\) be the Gram matrix
of \(B\).  Then
\[
  \det(\Lambda(B))^2
  =
  \det G
  =
  \prod_{i=1}^nd_i.
\]
Consequently,
\[
  \Delta(B)
  =
  \frac{\prod_{i=1}^nG_{ii}}{\det G}
  =
  \prod_{i=1}^n
  \frac{\lVert \vect{b}_i\rVert^2}{\lVert \vect{b}_i^*\rVert^2}.
\]

\section{Asymptotic Growth of the HKZ Orthogonality Defect}
\label{sec:asymptotic-growth}

We give the proof of Theorem~\ref{thm:asymptotic-growth} in this section.
First we obtain the asymptotic lower bound of $D_n$ by combining two HKZ-reduced blocks.

\begin{lemma}
\label{lem:all-half-block}
Let \(U\) be an \(r\)-dimensional Euclidean space with orthonormal basis
\(\vect{e}_1,\ldots,\vect{e}_r\), and define
\begin{equation}
  \vect{a}_i=\vect{e}_i-\frac12\sum_{k<i}\vect{e}_k,
  \qquad
  A_r=(\vect{a}_1,\ldots,\vect{a}_r).
  \label{eq:all-half-block}
\end{equation}
Then \(A_r\) is HKZ-reduced, with Gram--Schmidt vectors
\(\vect{a}_i^*=\vect{e}_i\), and
\begin{equation}
  \Delta(A_r)
  =
  \prod_{i=1}^r\left(1+\frac{i-1}{4}\right)
  =
  \frac{(r+3)!}{6\cdot4^r}.
  \label{eq:all-half-defect}
\end{equation}
\end{lemma}

\begin{proof}
Equation~\eqref{eq:all-half-block} is a Gram--Schmidt expansion in
the orthonormal vectors \(\vect{e}_i\).  Hence
\(\vect{a}_i^*=\vect{e}_i\) and
\(\mu_{i,k}=-1/2\) for \(k<i\); thus \(A_r\) is size-reduced.

Fix \(p\in\{1,\ldots,r\}\).  After projection away from
\(\vect{e}_1,\ldots,\vect{e}_{p-1}\), the remaining basis vectors are
\[
  \vect{a}_j^{(p)}
  =
  \vect{e}_j-\frac12\sum_{p\leq k<j}\vect{e}_k,
  \qquad p\leq j\leq r.
\]
For a nonzero vector
\(\vect{x}=\sum_{j=p}^r z_j\vect{a}_j^{(p)}\), let \(q\) be the
largest index with \(z_q\neq0\).  The \(\vect{e}_q\)-coordinate of
\(\vect{x}\) is \(z_q\), and therefore
\(\lVert \vect{x}\rVert^2\geq z_q^2\geq1\).  Since
\(\vect{a}_p^{(p)}=\vect{e}_p\), the first minimum of the projected lattice
is \(1\) at every
level.  Thus
\(A_r\) is HKZ-reduced.

Finally, \(\lVert \vect{a}_i\rVert^2=1+(i-1)/4\) and
\(\det(\Lambda(A_r))=1\).  Multiplication gives
\eqref{eq:all-half-defect}.
\end{proof}

\begin{lemma}
\label{lem:hanrot-stehle-block}
For every \(s\geq1\), there exists an HKZ-reduced basis
\(C_s=(\vect{c}_1,\ldots,\vect{c}_s)\) such that
\(\lVert\vect{c}_1\rVert=1\) and its squared Gram--Schmidt lengths are
\[
  d_j
  =
  \lVert \vect{c}_j^*\rVert^2
  =
  \frac{s-j+1}{s}
  \prod_{t=s-j+1}^{s-1}
  \bigl(e^{-6}(t+1)\bigr)^{-1/t}.
\]
Since \(\vect{c}_1\) is a shortest nonzero vector of \(\Lambda(C_s)\),
\begin{equation}
  \lVert \vect{x}\rVert^2\geq1
  \qquad
  \text{for every }\vect{0}\neq\vect{x}\in\Lambda(C_s).
  \label{eq:hs-minimum}
\end{equation}
Moreover, direct calculation gives
\begin{equation}
  \prod_{j=1}^s d_j
  =
  e^{6(s-1)}s^{-s}.
  \label{eq:hs-gso-product}
\end{equation}
\end{lemma}

\begin{proof}
See Hanrot and Stehl\'e \cite[Theorems~1--2 and Corollary~1]{hanrotstehle2008}.
\end{proof}

\begin{proposition}
\label{prop:coupled-hkz-basis}
Let \(r,s\geq1\), let \(A_r\) be the basis from
Lemma~\ref{lem:all-half-block}, embedded in a Euclidean space \(U\), and let \(C_s\) be the basis from
Lemma~\ref{lem:hanrot-stehle-block}, embedded in a Euclidean
space \(W\).  Place \(U\) and \(W\) orthogonally and, with a slight abuse of
notation, set
\[
  \vect{u}=\frac12\sum_{i=1}^r \vect{e}_i 
  \qquad\text{and}\qquad
  B_{r,s}
  =
  (\vect{a}_1,\ldots,\vect{a}_r,
  \vect{c}_1+\vect{u},\ldots,\vect{c}_s+\vect{u}).
\]
Then \(B_{r,s}\) is an \((r+s)\)-dimensional HKZ-reduced basis and
\begin{equation}
  \Delta(B_{r,s})
  \geq
  \frac{(r+3)!}{6\cdot4^r}
  \left(1+\frac r4\right)^s
  s^s e^{-6(s-1)}.
  \label{eq:coupled-defect-bound}
\end{equation}
\end{proposition}

\begin{proof}
The Gram--Schmidt vectors of \(B_{r,s}\) are
\[
  \vect{e}_1,\ldots,\vect{e}_r,
  \vect{c}_1^*,\ldots,\vect{c}_s^*.
\]
The coefficients are
\begin{align*}
  &\mu_{i,j}=-\frac{1}{2} \qquad(1\leq j<i\leq r),\\
  &\mu_{r+i,j}=\frac{1}{2} \qquad(1\leq j\leq r,\ 1\leq i\leq s),
\end{align*}
and those \(\mu_{r+i,r+j}\) with \(1\leq j<i\leq s\) are the corresponding
size-reduced coefficients of \(C_s\).  Thus \(B_{r,s}\) is
size-reduced.

For \(1\leq q\leq s\), the projected lattice
\(\Lambda_{r+q}(B_{r,s})\), together with its basis, coincides
with \(\Lambda_q(C_s)\) and its basis.  Hence the HKZ
shortest-vector condition holds at every level in the second block.

Now fix \(1\leq p\leq r\).  A basis of
\(\Lambda_p(B_{r,s})\) is
\[
  \vect{a}_i^{(p)}=\vect{e}_i-
  \frac12\sum_{p\leq k<i}\vect{e}_k
  \quad(p\leq i\leq r),
  \qquad
  \vect{c}_j+\vect{u}_p
  \quad(1\leq j\leq s),
\]
where \(\vect{u}_p=\frac12\sum_{i=p}^r \vect{e}_i\).  If an integral
combination of basis vectors of \(\Lambda_p(B_{r,s})\) has a
nonzero coefficient in the second block, its \(W\)-component is a nonzero
vector of \(\Lambda(C_s)\), so its squared norm is at least \(1\) by
\eqref{eq:hs-minimum}.  If all second-block coefficients vanish,
Lemma~\ref{lem:all-half-block} gives the same lower bound.  Since
\(\vect{a}_p^{(p)}=\vect{e}_p\) belongs to \(\Lambda_p(B_{r,s})\), the
minimum of this projected lattice is exactly \(1\).
All projected shortest-vector conditions follow, and \(B_{r,s}\) is
HKZ-reduced.

For the defect, \(\lVert \vect{u}\rVert^2=r/4\), and
\eqref{eq:hs-minimum} gives
\[
  \lVert \vect{c}_j+\vect{u}\rVert^2
  =
  \lVert \vect{c}_j\rVert^2+\lVert \vect{u}\rVert^2
  \geq
  1+\frac r4.
\]
The squared covolume is the product of the squared Gram--Schmidt lengths,
and hence, by \eqref{eq:hs-gso-product},
\[
  \det(\Lambda(B_{r,s}))^2
  =
  \prod_{j=1}^s d_j
  =
  e^{6(s-1)}s^{-s}.
\]
Combining this identity with \eqref{eq:all-half-defect} proves
\eqref{eq:coupled-defect-bound}.
\end{proof}

\begin{proposition}
\label{prop:asymptotic-lower-bound}
As \(n\to\infty\), one has
\[
  \log D_n\geq2n\log n-n\log\log n-(7+\log4)n+o(n).
\]
\end{proposition}

\begin{proof}
Define
\begin{equation}
  F_n(r)
  =
  \frac{(r+3)!}{6\cdot4^r}
  \left(1+\frac r4\right)^{n-r}
  (n-r)^{n-r}e^{-6(n-r-1)}.
  \label{eq:finite-term}
\end{equation}
Proposition~\ref{prop:coupled-hkz-basis} gives
\[
  D_n\geq F_n(r), \quad \forall \ 1\leq r\leq n-1.
\]
We next locate the maximum of \(F_n(r)\) asymptotically.
For \(1\leq r\leq n-2\), write \(s=n-r\),
set
\begin{equation}
  R_n(r)
  =
  \frac{F_n(r+1)}{F_n(r)}
  =
  \frac{e^6}{s}
  \left[
    \left(1+\frac1{r+4}\right)
    \left(1-\frac1s\right)
  \right]^{s-1}.
  \label{eq:consecutive-ratio}
\end{equation}
Write \(\Phi_n(r)=\log R_n(r)\) and \(a=r+4\).  Regard \(\Phi_n\)
as a function of \(s=n-r\).  Then
\[
  \frac{d^2\Phi_n}{ds^2}
  =
  \frac{2}{a(a+1)}
  +\frac1{s(s-1)}
  +\frac{(s-1)(2a+1)}{a^2(a+1)^2}
  >0.
\]
Hence \(\{r:R_n(r)<1\}\) is an interval.  Since \(F_n\) increases when
\(R_n(r)>1\) and decreases when \(R_n(r)<1\), its only possible global
maxima are its first internal peak and the endpoint \(r=n-1\).

Let
\[
  \rho_n
  =
  \min\{1\leq r\leq n-2:R_n(r)<1\}.
\]
For \(r\asymp n/\log n\), expansion of \eqref{eq:consecutive-ratio} gives
\[
  \Phi_n(r)
  =
  5-\log(n-r)+\frac{n-r}{r+4}+o(1).
\]
It is positive at \(r=\lfloor n/(2\log n)\rfloor\) and negative at
\(r=\lceil2n/\log n\rceil\).  Thus \(\rho_n\) exists for all sufficiently
large \(n\).  In this range, the explicit formula
\eqref{eq:consecutive-ratio} gives uniformly
\(\Phi_n(r+1)-\Phi_n(r)=O((\log n)^2/n)=o(1)\).  Hence the first crossing
satisfies
\[
  \frac{n-\rho_n}{\rho_n+4}
  =
  \log(n-\rho_n)-5+o(1).
\]
Consequently,
\[
  \rho_n
  =
  \frac{n}{\log n-4+o(1)}
\]
as \(n\to\infty\).

Stirling's formula applied to
\eqref{eq:finite-term} gives
\[
  \log F_n(\rho_n)
  =
  2n\log n-n\log\log n-(7+\log 4)n+o(n),
\]
whereas \(\log F_n(n-1)=n\log n+O(n)\).

Thus
\[
  \log D_n
  \geq
  \max\limits_r \log F_n(r)
  =
  2n\log n-n\log\log n-(7+\log 4)n+o(n). \qedhere
\]
\end{proof}

\begin{remark}
The coupling of the two HKZ blocks gives a stronger lower bound for \(D_n\)
than treating either block independently. The leading term of the lower bound given by both the all-half block and
Lemma~\ref{lem:hanrot-stehle-block} is
\( n\log n\), while our lower bound of \(\log\Delta(B_{r,s})\) has leading term \(2n\log n\).
\end{remark}

Next we obtain the asymptotic upper bound of $D_n$. 

\begin{lemma}
\label{lem:threshold-product}
Let \(B=(\vect{b}_1,\ldots,\vect{b}_\ell)\) be an \(\ell\)-dimensional HKZ-reduced basis. Denote
\[
  M_j=\max_{1\leq i\leq j}d_i,
  \qquad a_j=\frac{d_j}{M_j},
  \qquad s=\sum_{j=1}^{\ell}a_j.
\]
Then
\begin{equation}
  \log\prod_{j=1}^{\ell}a_j
  \geq-(\ell-s)\log\ell-\log\ell-\ell.
  \label{eq:threshold-product}
\end{equation}
\end{lemma}

\begin{proof}
\(a_1=1\) and \(0<a_j\leq1\). Fix \(0<t\leq1\), and let
\[
  1=k_1<\cdots<k_{N(t)}
\]
be all indices such that \(a_{k_p}\geq t\), where \(N(t)=\#\{j:a_j\geq t\}\). Set \(k_{N(t)+1}=\ell+1\) and
\[
  m_p=k_{p+1}-k_p\qquad(1\leq p\leq N(t)).
\]
Note that \(\sum_{p=1}^{N(t)}m_p=\ell\) and that \(M_j\) is constant for \(k_p\leq j<k_{p+1}\). 

For each \(p\), consider the lattice
\[
  \Lambda_{k_p,m_p}(B)
  =\bigl\langle\pi_{k_p}(\vect{b}_{k_p}),\ldots,
       \pi_{k_p}(\vect{b}_{k_{p+1}-1})\bigr\rangle_{\Z}.
\]
Then
\[
  \lambda_1(\Lambda_{k_p,m_p}(B))^2=d_{k_p},
  \qquad
  \det(\Lambda_{k_p,m_p}(B))^2
  =\prod_{j=k_p}^{k_{p+1}-1}d_j.
\]
By the classical estimate of the Hermite constant \(\gamma_{m_p}\leq m_p\) \cite[p.~20]{conwaysloane1993}, \cite[Corollary~4.1]{Zong2025}, we obtain
\[
  a_{k_p}=\frac{d_{k_p}}{M_{k_p}}
  \leq\gamma_{m_p}
  \left(\prod_{j=k_p}^{k_{p+1}-1}a_j\right)^{1/m_p}
  \leq m_p\left(\prod_{j=k_p}^{k_{p+1}-1}a_j\right)^{1/m_p}.
\]
Consequently,
\[
  \prod_{j=k_p}^{k_{p+1}-1}a_j
  \geq\left(\frac{a_{k_p}}{m_p}\right)^{m_p}
  \geq\left(\frac{t}{m_p}\right)^{m_p}.
\]
Thus
\[
  \log\prod_{j=1}^{\ell}a_j
  \geq\sum_{p=1}^{N(t)}m_p\log\frac{t}{m_p}
  =\ell\log t-\sum_{p=1}^{N(t)}m_p\log m_p.
\]

Note that for any convex function \(f\) on \([1,\infty)\) and integers \(2\leq a\leq b\), one has
\[
  f(a)+f(b)\leq f(a-1)+f(b+1).
\]
Thus by the convexity of \(f(x)=x\log x\), we have
\[
  m_p\log m_p + m_{N(t)}\log m_{N(t)} \leq \log 1 + (m_{N(t)}+m_p-1)\log (m_{N(t)}+m_p-1). 
\]
Applying to all $p\leq N(t)$ gives
\[
  \sum_{p=1}^{N(t)}m_p\log m_p
  \leq(\ell-N(t)+1)\log(\ell-N(t)+1)
  \leq(\ell-N(t)+1)\log\ell,
\]
and hence
\[
  \log\prod_{j=1}^{\ell}a_j
  \geq\ell\log t-(\ell-N(t)+1)\log\ell.
\]
Since 
\[
  \int_0^1 N(t)\,dt=\sum_{j=1}^{\ell}a_j=s,
\]
integration over \(0<t<1\) gives
\[
  \log\prod_{j=1}^{\ell}a_j
  \geq-\ell-(\ell+1)\log\ell
    +\log\ell\sum_{j=1}^{\ell}a_j
  =-(\ell-s)\log\ell-\log\ell-\ell. \qedhere
\]
\end{proof}

\begin{proposition}
\label{prop:secondary-upper-bound}
For every \(n\geq2\), one has
\begin{equation}
  \log D_n
  \leq2n\log n-n\log\log n-n\log4+4\log n.
  \label{eq:secondary-upper-bound}
\end{equation}
\end{proposition}

\begin{proof}
Let \(B\) be any \(n\)-dimensional HKZ-reduced basis, and set \(L=\log n\). Let \(M_i\) and \(a_i\) be as in Lemma~\ref{lem:threshold-product}, with \(\ell=n\), and set
\[
  S=\sum_{i=1}^{n}a_i.
\]
Then \(1\leq S\leq n\), and Lemma~\ref{lem:threshold-product} gives
\begin{equation}
  \log\prod_{i=1}^{n}a_i
  \geq-(n-S)L-L-n.
  \label{eq:normalized-product-lower-bound}
\end{equation}
Size reduction and \(M_j\leq M_i\) for \(j\leq i\) imply
\begin{align*}
  \frac{\lVert\vect{b}_i\rVert^2}{M_i}
  &=a_i+\sum_{j<i}\mu_{i,j}^2\frac{M_j}{M_i}a_j\\
  &\leq a_i+\frac14\sum_{j<i}a_j
   \leq\frac{S+3a_i}{4}
   \leq\frac{S+3}{4}.
\end{align*}
Thus 
\[
  \Delta(B)\leq
  \frac{((S+3)/4)^n}{\prod_{i=1}^{n}a_i}.
\]
By \eqref{eq:normalized-product-lower-bound}, we have
\[
  \log\Delta(B)
  \leq nL+n-n\log4+L+n\log(S+3)-LS.
\]
Since \(L>0\), the concave function \(n\log(S+3)-LS\) on \(S>-3\) attains its maximum at \(S=n/L-3\), with its maximum being
\[
  n\log\frac nL-n+3L.
\]
Therefore, 
\[
  \log\Delta(B)
  \leq2nL-n\log L-n\log4+4L.
\]
Taking the supremum over all such \(B\) proves \eqref{eq:secondary-upper-bound}.
\end{proof}

By combining Proposition~\ref{prop:asymptotic-lower-bound} and Proposition~\ref{prop:secondary-upper-bound}, one proves Theorem~\ref{thm:asymptotic-growth}.

\section{The Exact Four-Dimensional Constant \texorpdfstring{\(D_4\)}{D4}}
\label{sec:rank-four}

\begin{lemma}
\label{lem:rank-four-lower-bound}
There exists a four-dimensional HKZ-reduced basis \(B\) such that
\[
  \Delta(B)=\frac{4375}{1024}.
\]
Consequently, \(D_4\geq4375/1024\).
\end{lemma}

\begin{proof}
We explicitly construct \(B\).  In \(\R^4\), let
\[
  \vect{b}_1=(2\sqrt5,0,0,0),\qquad
  \vect{b}_2=(-\sqrt5,4,0,0),
\]
\[
  \vect{b}_3=(\sqrt5,-2,4,0),\qquad
  \vect{b}_4=(-\sqrt5,2,2,2\sqrt3),
\]
and set \(B=(\vect{b}_1,\vect{b}_2,\vect{b}_3,\vect{b}_4)\).
Then every lattice vector \(\sum_{i=1}^{4}z_i\vect{b}_i\in\Lambda(B)\)
has squared norm
\begin{align*}
  \left\lVert\sum_{i=1}^{4}z_i\vect{b}_i\right\rVert^2
  & =Q(z_1,z_2,z_3,z_4)\\
  & =20\left(z_1-\frac{z_2}{2}
              +\frac{z_3}{2}-\frac{z_4}{2}\right)^2\\
  &\quad+16\left(z_2-\frac{z_3}{2}+\frac{z_4}{2}\right)^2\\
  &\quad+16\left(z_3+\frac{z_4}{2}\right)^2+12z_4^2.
\end{align*}
Thus \(B\) is size-reduced.

It remains to verify the shortest-vector
condition in each projected lattice.  The corresponding projected norm
forms are
\[
  Q_4=12z_4^2,
  \qquad
  Q_3=4(2z_3+z_4)^2+12z_4^2,
  \qquad
  Q_2=4(2z_2-z_3+z_4)^2+Q_3.
\]
Clearly,
\[
  \min\limits_{z_4\in\Z\setminus\{0\}}
  Q_4(z_4)=12=d_4.
\]
By considering separately \(|z_4|=0\), \(|z_4|=1\), and
\(|z_4|\geq2\), we obtain
\[
  \min_{\vect{z}\in\Z^{2}\setminus\{\vect{0}\}}
  Q_3(\vect{z})=16=d_3,
  \quad\text{attained, for example, at }(z_3,z_4)=(0,1).
\]
Now for \((z_2,z_3,z_4)\neq(0,0,0)\), if
\((z_3,z_4)\neq(0,0)\), then \(Q_2\geq Q_3\geq16\).  Otherwise,
\(z_2\neq0\) and \(Q_2=16z_2^2\geq16\).  Hence
\[
\begin{gathered}
  \min_{\vect{z}\in\Z^{3}\setminus\{\vect{0}\}}
  Q_2(\vect{z})=16=d_2,
  \\
  \text{attained if and only if }
  (z_2,z_3,z_4)=\pm(1,0,0),\ \pm(1,1,-1).
\end{gathered}
\]
Finally, write
\[
  Q=5(2z_1-z_2+z_3-z_4)^2+Q_2.
\]
If \((z_2,z_3,z_4)=(0,0,0)\), then a nonzero integer vector has
\(z_1\neq0\), and \(Q=20z_1^2\geq20\).  Otherwise \(Q_2\geq16\).
Moreover, \(Q_2\) is a multiple of \(4\), so \(Q_2>16\) implies
\(Q_2\geq20\).  It remains only to consider \(Q_2=16\).  Direct
substitution of the four possibilities displayed above gives \(Q\geq21\).
Therefore,
\[
  \min_{\vect{z}\in\Z^{4}\setminus\{\vect{0}\}}
  Q(\vect{z})=20=d_1,
  \quad\text{attained, for example, at }(1,0,0,0).
\]

These calculations show that \(B\) is HKZ-reduced by
Definition~\ref{def:hkz}.  Moreover, direct calculation gives
\[
  \Delta(B)
  =
  \frac{20\cdot21\cdot25\cdot25}
       {20\cdot16\cdot16\cdot12}
  =
  \frac{4375}{1024}. \qedhere
\]
\end{proof}

Set
\[
  T=\frac{4375}{1024}.
\]
To prove the upper bound, normalize \(d_1=1\) and write the norm form of a
four-dimensional size-reduced basis \(B\) as
\[
  Q(\vect{z})
  =
  \sum_{j=1}^4d_j
  \left(z_j+\sum_{k>j}\mu_{k,j}z_k\right)^2,
  \qquad |\mu_{k,j}|\leq\frac12.
\]
Write the parameter vector
\[
  \vect{p}=
  (d_2,d_3,d_4,
  \mu_{2,1},\mu_{3,1},\mu_{3,2},
  \mu_{4,1},\mu_{4,2},\mu_{4,3})^{\mathsf T}.
\]
In these parameters the orthogonality defect is
\begin{align*}
  \delta(\vect{p})={}&
  \left(1+\frac{\mu_{2,1}^2}{d_2}\right)
  \left(1+
    \frac{\mu_{3,1}^2+d_2\mu_{3,2}^2}{d_3}\right)\\
  &\times
  \left(1+
    \frac{\mu_{4,1}^2+d_2\mu_{4,2}^2+d_3\mu_{4,3}^2}{d_4}\right)
  =\Delta(B).
\end{align*}

The classical KZ inequalities of Korkine and Zolotareff
\cite[pp.~372--376]{korkinzolotarev1873} imply that every HKZ-reduced form
satisfies
\begin{equation}
  d_2\geq\frac34,\qquad
  d_3\geq\frac34d_2,\quad d_3\geq\frac23,\qquad
  d_4\geq\frac34d_3,\quad d_4\geq\frac23d_2,\quad
  d_4\geq\frac12.
  \label{eq:rank-four-kz-inequalities}
\end{equation}
Changing basis-vector signs preserves both HKZ reduction and orthogonality
defect.  We may therefore assume
\(\mu_{2,1},\mu_{3,2},\mu_{4,3}\leq0\).
Novikova
\cite[Theorem~1 and Tables~1--3]{novikova1983}
gave a finite characterization that replaces the infinitely many
shortest-vector conditions of HKZ 
reduction, in dimension four, by the following \(21\)
inequalities:
\begin{equation}
  Q_i(\vect{z})\geq d_i
  \qquad
  (\vect{z}\in X_{5-i},\ i=1,2,3).
  \label{eq:rank-four-finite-kz-constraints}
\end{equation}
Here
\[
  X_2=\{(0,1)\},
  \qquad
  X_3=\{(0,1,0),(0,0,1),(0,1,1),(1,1,1)\},
\]
and
\(X_4=\{(v_1,v_2,v_3,0):(v_1,v_2,v_3)\in X_3\}\cup X_4^\circ\), where
\[
\begin{split}
  X_4^\circ=\{&
  (-1,-1,0,1),(0,-1,0,1),(0,0,0,1),(0,1,0,1),\\
  &(1,1,0,1),(-1,-1,1,1),(0,-1,1,1),(-1,0,1,1),\\
  &(0,0,1,1),(1,0,1,1),(0,1,1,1),(1,1,1,1)\}.
\end{split}
\]
Pendavingh and van Zwam
\cite[Sec.~2, Theorem~2.1]{pendavinghvanzwam2007} restated this finite
characterization and checked the low-dimensional systems; see also van Zwam
\cite{vanzwam2006} for implementation details and numerical data.
Under this sign convention, the size-reduction conditions, \eqref{eq:rank-four-kz-inequalities}, and \eqref{eq:rank-four-finite-kz-constraints} provide a finite characterization of the parameter vectors \(\vect{p}\) that represent normalized four-dimensional HKZ-reduced forms.

Put
\[
  I_+=\left[0,\frac12\right],
  \qquad
  I_-=\left[-\frac12,0\right].
\]
For \(\sigma=(\sigma_1,\sigma_2,\sigma_3)\in\{+,-\}^3\), let
\(\mathcal C_\sigma\) be the rational box in the coordinates of
\(\vect{p}\) defined by
\begin{equation}
  \frac34\leq d_2\leq2,\qquad
  \frac23\leq d_3\leq2,\qquad
  \frac12\leq d_4\leq2,
  \label{eq:rank-four-compact-box}
\end{equation}
\[
  \mu_{2,1},\mu_{3,2},\mu_{4,3}\in I_-,
  \qquad
  \mu_{3,1}\in I_{\sigma_1},\quad
  \mu_{4,1}\in I_{\sigma_2},\quad
  \mu_{4,2}\in I_{\sigma_3}.
\]

\begin{lemma}
\label{lem:rank-four-compact-reduction}
Let \(B\) be a normalized four-dimensional HKZ-reduced basis satisfying
\(\Delta(B)\geq T\).  After changes of basis-vector signs, its parameter
vector $\vect{p}$ belongs to one of the eight boxes \(\mathcal C_\sigma\).
\end{lemma}

\begin{proof}
It remains only to establish the bounds on \(d_2,d_3,d_4\)
in \eqref{eq:rank-four-compact-box}. The lower bounds 
follow directly from
\eqref{eq:rank-four-kz-inequalities}. For the upper bounds, size reduction gives
\[
  \Delta(B)\leq
  \left(1+\frac1{4d_2}\right)
  \left(1+\frac{1+d_2}{4d_3}\right)
  \left(1+\frac{1+d_2+d_3}{4d_4}\right).
\]
Together with \eqref{eq:rank-four-kz-inequalities}, for the three separate cases 
\(d_2\geq2\), \(d_2<2\leq d_3\), and
\(d_2,d_3<2\leq d_4\), we have
\[
  \Delta(B)\leq
  \frac{1431}{384},\quad
  \frac{583}{144},\quad
  \frac{325}{96}.
\]
The last estimate also uses \(D_3=25/12\). Since all three upper bounds are 
strictly smaller than \(T\), no normalized four-dimensional HKZ-reduced form 
with defect at least \(T\) can fall into any of these cases. 
\end{proof}

Let \(\varepsilon_0=1/1024\).  For
\(\eta,\theta\in\{-1,1\}\), write
\begin{equation}
  \vect{p}_{\eta,\theta}
  =\left(
    \frac45,\frac45,\frac35,
    -\frac12,\frac\eta2,-\frac12,
    \frac\theta2,-\frac12,-\frac12
  \right)^{\mathsf T} \in \mathcal{C_\sigma}.
  \label{eq:rank-four-corners}
\end{equation}
Here \(\sigma=(\operatorname{sgn}\eta,\operatorname{sgn}\theta,-)\), with the first two entries interpreted as \(+\) or \(-\).
Let \(\mathcal N_{\eta,\theta}\) be the intersection of \(\mathcal C_\sigma\)
 with the closed coordinate box of radius
\(\varepsilon_0\) centered at \(\vect{p}_{\eta,\theta}\).

\begin{lemma}
\label{lem:rank-four-rational-cover}
Every HKZ-feasible point in the union of the eight boxes
\(\mathcal C_\sigma\) that satisfies \(\delta(\vect{p})\geq T\) belongs to
one of the four boxes \(\mathcal N_{\eta,\theta}\).
\end{lemma}

\begin{proof}
The proof is computer-assisted. We recursively bisect each of the eight
initial boxes \(\mathcal C_\sigma\).  For a rational subbox
\(\mathcal C\), exact interval evaluation gives the upper bound
\begin{align*}
  \mathcal U(\mathcal C)={}&
  \left(1+\frac{\overline{\mu_{2,1}^2}}{\underline d_2}\right)
  \left(1+
    \frac{\overline{\mu_{3,1}^2}
    +\overline d_2\,\overline{\mu_{3,2}^2}}
         {\underline d_3}\right)\\
  &\times
  \left(1+
    \frac{\overline{\mu_{4,1}^2}
    +\overline d_2\,\overline{\mu_{4,2}^2}
    +\overline d_3\,\overline{\mu_{4,3}^2}}
         {\underline d_4}\right)
\end{align*}
for \(\delta(\vect{p})\) throughout \(\mathcal C\).  Here an underline
denotes the lower endpoint of an interval, while an overline denotes the
required upper endpoint or the maximum of a squared coefficient on that
interval.

Each subbox is processed as follows:
\begin{quote}
\small
\begin{enumerate}
  \item If \(\mathcal U(\mathcal C)<T\), record \(\mathcal C\) as below
        the target, and discard \(\mathcal C\). 
  \item If interval evaluation proves that one of the six inequalities
        \eqref{eq:rank-four-kz-inequalities} or one of the finite constraints
        \eqref{eq:rank-four-finite-kz-constraints} fails throughout
        \(\mathcal C\), record \(\mathcal C\) as HKZ-infeasible, and discard \(\mathcal C\). 
  \item If \(\mathcal C\) is contained in one of the four boxes
        \(\mathcal N_{\eta,\theta}\), record it as localized.
  \item Otherwise, bisect \(\mathcal C\) at the midpoint of a coordinate interval whose width divided by the width of its full allowed range before division into the eight sign boxes is largest, and process both children in the same way.
\end{enumerate}
\end{quote}

The first column of the following table gives the initial sign pattern
\(\sigma\).  The remaining columns count the terminal subboxes
classified by the procedure.
\[
\begin{array}{c|r|r|r|r}
\text{region } \mathcal{C}_{\sigma}&\text{HKZ-infeasible}&
\text{below target}&\text{localized}&\text{unresolved}\\ \hline
\mathtt{+++}& 1329&  3364&  0&0\\
\mathtt{++-}&28594& 50507&255&0\\
\mathtt{+-+}& 1143&  2919&  0&0\\
\mathtt{+--}&28854& 51313&255&0\\
\mathtt{-++}& 1143&  2937&  0&0\\
\mathtt{-+-}&29140& 51719&255&0\\
\mathtt{--+}& 1329&  3368&  0&0\\
\mathtt{---}&29008& 50877&255&0\\ \hline
\text{total}&120540&217004&1020&0
\end{array}
\]
The unresolved count is zero for every initial box.  Consequently, the
terminal subboxes form a complete finite cover of $\mathcal{C}_{\sigma}$, and every HKZ-feasible point
with defect at least \(T\) is contained in one of the four local boxes.  
The code implementing this exact
verification is included in Appendix~\ref{app:rational-covering-code}.
\end{proof}

\begin{lemma}
\label{lem:rank-four-local-maximum}
For every \(\eta,\theta\in\{-1,1\}\), an HKZ-feasible point in
\(\mathcal N_{\eta,\theta}\) satisfies
\[
  \delta(\vect{p})\leq T.
\]
Equality holds only at the center \(\vect{p}_{\eta,\theta}\).
\end{lemma}

\begin{proof}
We treat the four boxes simultaneously.  In
\(\mathcal N_{\eta,\theta}\), write
\[
  d_2=\frac45+\beta,\qquad
  d_3=\frac45+\chi,\qquad
  d_4=\frac35+\tau,
\]
and
\[
\begin{gathered}
  \mu_{2,1}=-\frac12+x_1,\qquad
  \mu_{3,1}=\eta\left(\frac12-x_2\right),\qquad
  \mu_{3,2}=-\frac12+x_3,\\
  \mu_{4,1}=\theta\left(\frac12-x_4\right),\qquad
  \mu_{4,2}=-\frac12+x_5,\qquad
  \mu_{4,3}=-\frac12+x_6.
\end{gathered}
\]
Thus \(0\leq x_j\leq\varepsilon_0\) for \(1\leq j\leq6\),
\(|\beta|,|\chi|,|\tau|\leq\varepsilon_0\), and the center \(\vect{p}_{\eta,\theta}\) of \(\mathcal N_{\eta,\theta}\) 
corresponds to
\[
  \vect{y}=(\beta,\chi,\tau,x_1,\ldots,x_6)^{\mathsf T}=\vect{0}.
\]
Write
\begin{align*}
  N_2=\lVert\vect{b}_2\rVert^2&=d_2+\mu_{2,1}^2,\\
  N_3=\lVert\vect{b}_3\rVert^2&=d_3+\mu_{3,1}^2+d_2\mu_{3,2}^2,\\
  N_4=\lVert\vect{b}_4\rVert^2&=d_4+\mu_{4,1}^2+d_2\mu_{4,2}^2+d_3\mu_{4,3}^2,
\end{align*}
and define
\[
  f(\vect{y})=\log\delta(\vect{p})
  =\sum_{i=2}^4\bigl(\log N_i-\log d_i\bigr).
\]

We use three of the finite HKZ constraints:
\[
  g_1=Q_3(0,1)-d_3,\qquad
  g_2=Q_2(1,1,1)-d_2,
\]
and
\[
  g_3=
  \begin{cases}
    Q_1(1,1,1,0)-1,&\eta=-1,\\
    Q_1(0,1,1,0)-1,&\eta=1.
  \end{cases}
\]
$g_i$ are functions of $\vect{y}$. Every HKZ-feasible point satisfies \(g_i\geq0\), and equality of all three constraints
holds at \(\vect{y}=\vect{0}\). Direct expansion gives \(g_i=h_i+E_i\), where the linear parts $h_i$ at the center are
\[
  h_1=-\frac34\chi+\tau-\frac45x_6,\qquad
  h_2=-\beta+\frac14\chi+\tau+\frac45x_6,\qquad
  h_3=\frac14\beta+\chi+\frac45x_3, 
\]
and each \(E_i\) contains only terms of degree at least
two. Also, 
\begin{equation}
  \sum_{i=1}^3|E_i|
  \leq\frac{98}{5}\varepsilon_0\lVert\vect{y}\rVert_1.
  \label{eq:rank-four-local-constraint-remainder}
\end{equation}

Consider the first-order Taylor term of $f(\vect{y})$
\begin{align*}
  \mathcal L(\vect{y})
  ={}&\nabla f(\vect{0})^{\mathsf T}\vect{y}\\
  ={}&
  \frac{43}{420}\beta-\frac14\chi-\frac{13}{15}\tau\\
  &-\frac{20}{21}x_1-\frac45x_2-\frac{16}{25}x_3
  -\frac45x_4-\frac{16}{25}x_5-\frac{16}{25}x_6\\
  ={}&
  -\frac{127}{210}h_1-\frac{11}{42}h_2-\frac{67}{105}h_3\\
  &-\frac{20}{21}x_1-\frac45x_2-\frac{68}{525}x_3
  -\frac45x_4-\frac{16}{25}x_5-\frac{32}{35}x_6.
\end{align*}
Thus the directional derivative of \(f\) at \(\vect{0}\) is negative in every nonzero direction satisfying the linearized conditions \(h_i\geq0\) and \(x_j\geq0\). We next make this statement quantitative for the nonlinear constraints.

Define \(h_i^+=\max\{h_i,0\}\), \(h_i^-=\max\{-h_i,0\}\), and
\(\vect{h}_\pm=(h_1^\pm,h_2^\pm,h_3^\pm)^{\mathsf T}\), so that
\(\vect{h}=\vect{h}_+-\vect{h}_-\).  Also put
\(X=\sum_{j=1}^6x_j\).  Since \(g_i=h_i+E_i\geq0\), one has
\(h_i^-\leq|E_i|\).  Therefore
\eqref{eq:rank-four-local-constraint-remainder} yields
\begin{equation}
  \lVert\vect{h}_-\rVert_1
  \leq\frac{98}{5}\varepsilon_0\lVert\vect{y}\rVert_1.
  \label{eq:rank-four-negative-part}
\end{equation}

The three linear forms \(h_i\) determine \(\beta,\chi,\tau\) once
\(x_3,x_6\) are fixed.  More precisely,
\[
  \vect{h}
  =A(\beta,\chi,\tau)^{\mathsf T}
   +\frac45(-x_6,x_6,x_3)^{\mathsf T},
  \qquad
  A=
  \begin{pmatrix}
    0&-3/4&1\\
    -1&1/4&1\\
    1/4&1&0
  \end{pmatrix}.
\]
For the induced matrix \(1\)-norm,
\(\lVert A^{-1}\rVert_1=11/5\), while
\[
  \left\lVert\frac45(-x_6,x_6,x_3)^{\mathsf T}\right\rVert_1
  \leq\frac85X.
\]
It follows that
\begin{align*}
  \lVert\vect{y}\rVert_1&=\lVert(\beta,\chi,\tau)^{\mathsf T}\rVert_1+X\\
  &\leq\frac{11}{5}
  \left(\lVert\vect{h}_+\rVert_1+
        \lVert\vect{h}_-\rVert_1+\frac85X\right)+X\\
  &\leq\frac{113}{25}
  \bigl(\lVert\vect{h}_+\rVert_1+
        \lVert\vect{h}_-\rVert_1+X\bigr).
\end{align*}
The smallest absolute value of coefficients in the displayed decomposition of
\(\mathcal L(\vect{y})\) is \(68/525\), and the largest absolute value of coefficients of
\(h_i\) is \(67/105\).  Hence
\begin{align*}
  \mathcal L(\vect{y})
  &\leq
  -\frac{68}{525}
   \bigl(\lVert\vect{h}_+\rVert_1+X\bigr)
  +\frac{67}{105}\lVert\vect{h}_-\rVert_1\\
  &\leq
  -\frac{68}{2373}\lVert\vect{y}\rVert_1
  +\left(\frac{68}{525}+\frac{67}{105}\right)
   \lVert\vect{h}_-\rVert_1.
\end{align*}
Using \eqref{eq:rank-four-negative-part}, we obtain
\begin{equation}
  \mathcal L(\vect{y})
  \leq
  \left[
    -\frac{68}{2373}
    +\left(\frac{68}{525}+\frac{67}{105}\right)
      \frac{98}{5}\varepsilon_0
  \right]\lVert\vect{y}\rVert_1.
  \label{eq:rank-four-linear-decrease}
\end{equation}

Direct differentiation of the explicit formula for \(f\) shows that the
maximum absolute row sum of its Hessian is less than \(21\) throughout the
four local boxes.  Apply the one-variable Taylor theorem to
\(\phi(t)=f(t\vect{y})\).  The segment from \(\vect{0}\) to \(\vect{y}\) remains
inside the same local box, so the Hessian bound and
\(\lVert\vect{y}\rVert_\infty\leq\varepsilon_0\) give
\begin{equation}
  |f(\vect{y})-f(\vect{0})-\mathcal L(\vect{y})|
  \leq\frac{21}{2}\varepsilon_0\lVert\vect{y}\rVert_1.
  \label{eq:rank-four-objective-remainder}
\end{equation}
At every candidate center,
\(f(\vect{0})=\log T\).  Combining
\eqref{eq:rank-four-linear-decrease} and
\eqref{eq:rank-four-objective-remainder}, and substituting
\(\varepsilon_0=1/1024\), gives
\[
  \log\frac{\delta(\vect{p})}{T}
  \leq
  -\frac{751077}{202496000}\lVert\vect{y}\rVert_1\leq0.
\]
Equality is possible only when \(\vect{y}=\vect{0}\), which is the center
\(\vect{p}_{\eta,\theta}\).
\end{proof}

\begin{proof}[Proof of Theorem~\ref{thm:rank-four-exact}]
Lemma~\ref{lem:rank-four-lower-bound} shows that \(T\) is attained by an
HKZ-reduced basis.  Now let \(B\) be any four-dimensional HKZ-reduced
basis satisfying \(\Delta(B)\geq T\), and normalize \(d_1=1\).  After
suitable changes of basis-vector signs,
Lemma~\ref{lem:rank-four-compact-reduction} places the parameter vector
\(\vect{p}\) of \(B\) in one of the eight boxes \(\mathcal C_\sigma\).
Lemma~\ref{lem:rank-four-rational-cover} then places \(\vect{p}\) in one
of the four local boxes \(\mathcal N_{\eta,\theta}\), where
Lemma~\ref{lem:rank-four-local-maximum} gives \(\Delta(B)\leq T\) in these boxes.
Hence
\[
  D_4=T=\frac{4375}{1024}.
\]

Finally, the four equality cases identified by
Lemmas~\ref{lem:rank-four-compact-reduction}--\ref{lem:rank-four-local-maximum}
are related by unimodular changes of basis and all correspond, up to
similarity, to the lattice constructed in
Lemma~\ref{lem:rank-four-lower-bound}; hence the maximizing lattice for $D_4$ is
unique up to similarity.
\end{proof}

\begin{remark}
In every fixed dimension, HKZ reduction can be characterized by finitely many
inequalities, and Novikova \cite[Theorem~1 and Tables~1--3]{novikova1983}
gave explicit finite systems through dimension eight.  Thus the present 
approach may in principle be applied to higher-dimensional \(D_n\).
\end{remark}

\subsection*{Acknowledgements}
I am grateful to Professor Gang Tian for his
constant support and to Professor Chuanming Zong for introducing me to
lattice-based cryptography and, in particular, to this problem.  
During the preparation of this manuscript, the author used OpenAI Codex to
assist with organization and language editing.  The author reviewed and
verified the resulting text and assumes full responsibility for the content.

\clearpage
\appendix
\section{Code for the Computer-Assisted Verification}
\label{app:rational-covering-code}

The following self-contained Python~3 program implements the exact rational
covering used in Lemma~\ref{lem:rank-four-rational-cover}.  It uses only
modules from the Python standard library.  Running it without command-line
arguments verifies all eight initial boxes and prints the terminal counts
reported in the proof.  All interval endpoints and comparisons are represented
by \texttt{fractions.Fraction}; the program performs no floating-point
operations.

\renewcommand{\theFancyVerbLine}{%
  \appendixcodenumberfont\arabic{FancyVerbLine}}
\begin{Verbatim}[
  formatcom={\fontsize{6.5pt}{7.2pt}\selectfont\ttfamily},
  numbers=left,
  numbersep=6pt
]
#!/usr/bin/env python3
"""Exact rational covering certificate for the rank-four HKZ defect.

The verifier proves that every normalized rank-four HKZ form with defect at
least 4375/1024 lies in one of four radius-1/1024 neighborhoods of the
candidate corners.  The complementary local inequality and equality
classification are proved analytically in the paper.

No floating-point operation is used in this file.
"""

from __future__ import annotations

import argparse
import heapq
import itertools
import json
from dataclasses import dataclass
from fractions import Fraction
from pathlib import Path

Q = Fraction
TARGET = Q(4375, 1024)
LOCAL_RADIUS = Q(1, 1024)

X2 = [(0, 1)]
X3_NEW = [(0, 0, 1), (0, 1, 1), (1, 1, 1)]
X3 = [(values[0], values[1], 0) for values in X2] + X3_NEW
X4_NEW = [
    (-1, -1, 0, 1),
    (0, -1, 0, 1),
    (0, 0, 0, 1),
    (0, 1, 0, 1),
    (1, 1, 0, 1),
    (-1, -1, 1, 1),
    (0, -1, 1, 1),
    (-1, 0, 1, 1),
    (0, 0, 1, 1),
    (1, 0, 1, 1),
    (0, 1, 1, 1),
    (1, 1, 1, 1),
]
X4 = [tuple(values) + (0,) for values in X3] + X4_NEW

D2, D3, D4, M21, M31, M32, M41, M42, M43 = range(9)
MU_POSITIONS = {
    (1, 0): M21,
    (2, 0): M31,
    (2, 1): M32,
    (3, 0): M41,
    (3, 1): M42,
    (3, 2): M43,
}
SCALES = (Q(5, 4), Q(4, 3), Q(3, 2), Q(1, 2), Q(1), Q(1, 2), Q(1), Q(1), Q(1, 2))


@dataclass(frozen=True)
class Box:
    lower: tuple[Fraction, ...]
    upper: tuple[Fraction, ...]
    depth: int = 0


def add_interval(
    left: tuple[Fraction, Fraction],
    right: tuple[Fraction, Fraction],
) -> tuple[Fraction, Fraction]:
    return left[0] + right[0], left[1] + right[1]


def mul_interval(
    left: tuple[Fraction, Fraction],
    right: tuple[Fraction, Fraction],
) -> tuple[Fraction, Fraction]:
    products = (
        left[0] * right[0],
        left[0] * right[1],
        left[1] * right[0],
        left[1] * right[1],
    )
    return min(products), max(products)


def square_interval(
    interval: tuple[Fraction, Fraction],
) -> tuple[Fraction, Fraction]:
    lower, upper = interval
    minimum = Q(0) if lower <= 0 <= upper else min(lower * lower, upper * upper)
    return minimum, max(lower * lower, upper * upper)


def coefficient_interval(
    box: Box, first: int, values: tuple[int, ...], j: int
) -> tuple[Fraction, Fraction]:
    x = [0, 0, 0, 0]
    for index, value in enumerate(values, start=first):
        x[index] = value
    result = (Q(x[j]), Q(x[j]))
    for k in range(j + 1, 4):
        multiplier = x[k]
        if multiplier == 0:
            continue
        position = MU_POSITIONS[(k, j)]
        term = mul_interval(
            (Q(multiplier), Q(multiplier)),
            (box.lower[position], box.upper[position]),
        )
        result = add_interval(result, term)
    return result


def projected_constraint_upper(
    box: Box, first: int, values: tuple[int, ...]
) -> Fraction:
    d_intervals = (
        (Q(1), Q(1)),
        (box.lower[D2], box.upper[D2]),
        (box.lower[D3], box.upper[D3]),
        (box.lower[D4], box.upper[D4]),
    )
    total = (Q(0), Q(0))
    for j in range(first, 4):
        coefficient = coefficient_interval(box, first, values, j)
        total = add_interval(
            total,
            mul_interval(d_intervals[j], square_interval(coefficient)),
        )
    return total[1] - d_intervals[first][0]


def infeasible(box: Box) -> bool:
    if box.upper[D2] < Q(3, 4):
        return True
    if box.upper[D3] < Q(3, 4) * box.lower[D2]:
        return True
    if box.upper[D3] < Q(2, 3):
        return True
    if box.upper[D4] < Q(3, 4) * box.lower[D3]:
        return True
    if box.upper[D4] < Q(2, 3) * box.lower[D2]:
        return True
    if box.upper[D4] < Q(1, 2):
        return True
    for first, vectors in ((0, X4), (1, X3), (2, X2)):
        for values in vectors:
            if projected_constraint_upper(box, first, values) < 0:
                return True
    return False


def square_upper(box: Box, position: int) -> Fraction:
    lower, upper = box.lower[position], box.upper[position]
    return max(lower * lower, upper * upper)


def defect_upper(box: Box) -> Fraction:
    factor2 = Q(1) + square_upper(box, M21) / box.lower[D2]
    factor3 = Q(1) + (
        square_upper(box, M31)
        + box.upper[D2] * square_upper(box, M32)
    ) / box.lower[D3]
    factor4 = Q(1) + (
        square_upper(box, M41)
        + box.upper[D2] * square_upper(box, M42)
        + box.upper[D3] * square_upper(box, M43)
    ) / box.lower[D4]
    return factor2 * factor3 * factor4


def split_position(box: Box) -> int:
    return max(
        range(9),
        key=lambda position: (
            (box.upper[position] - box.lower[position]) / SCALES[position],
            -position,
        ),
    )


def split(box: Box, position: int) -> tuple[Box, Box]:
    lower = list(box.lower)
    upper = list(box.upper)
    midpoint = (lower[position] + upper[position]) / 2
    left_upper = upper.copy()
    left_upper[position] = midpoint
    right_lower = lower.copy()
    right_lower[position] = midpoint
    return (
        Box(tuple(lower), tuple(left_upper), box.depth + 1),
        Box(tuple(right_lower), tuple(upper), box.depth + 1),
    )


def initial_box(signs: str) -> Box:
    lower = [Q(3, 4), Q(2, 3), Q(1, 2), -Q(1, 2), -Q(1, 2), -Q(1, 2), -Q(1, 2), -Q(1, 2), -Q(1, 2)]
    upper = [Q(2), Q(2), Q(2), Q(0), Q(1, 2), Q(0), Q(1, 2), Q(1, 2), Q(0)]
    for position, sign in zip((M31, M41, M42), signs):
        if sign == "+":
            lower[position] = Q(0)
        else:
            upper[position] = Q(0)
    return Box(tuple(lower), tuple(upper))


def inside_candidate_neighborhood(box: Box, signs: str) -> bool:
    if signs[2] != "-":
        return False
    center = [Q(4, 5), Q(4, 5), Q(3, 5), -Q(1, 2), Q(0), -Q(1, 2), Q(0), Q(0), -Q(1, 2)]
    for position, sign in zip((M31, M41, M42), signs):
        center[position] = Q(1, 2) if sign == "+" else -Q(1, 2)
    return all(
        box.lower[position] >= value - LOCAL_RADIUS
        and box.upper[position] <= value + LOCAL_RADIUS
        for position, value in enumerate(center)
    )


def verify_region(signs: str, max_boxes: int, report_every: int) -> dict[str, int | str]:
    root = initial_box(signs)
    queue: list[tuple[Fraction, int, Box]] = [(-defect_upper(root), 0, root)]
    serial = 0
    processed = 0
    pruned_infeasible = 0
    pruned_objective = 0
    retained_local = 0
    maximum_depth = 0

    while queue:
        if processed >= max_boxes:
            raise RuntimeError(f"{signs}: exceeded max-boxes={max_boxes}")
        negative_bound, _, box = heapq.heappop(queue)
        processed += 1
        maximum_depth = max(maximum_depth, box.depth)
        if -negative_bound < TARGET:
            pruned_objective += 1
            continue
        if infeasible(box):
            pruned_infeasible += 1
            continue
        if inside_candidate_neighborhood(box, signs):
            retained_local += 1
            continue
        position = split_position(box)
        for child in split(box, position):
            serial += 1
            heapq.heappush(queue, (-defect_upper(child), serial, child))
        if report_every and processed % report_every == 0:
            print(
                f"{signs}: processed={processed} open={len(queue)}"
            )

    return {
        "signs": signs,
        "processed": processed,
        "pruned_infeasible": pruned_infeasible,
        "pruned_objective": pruned_objective,
        "retained_local": retained_local,
        "maximum_depth": maximum_depth,
    }


def main() -> None:
    parser = argparse.ArgumentParser()
    parser.add_argument(
        "--signs",
        action="append",
        choices=["".join(s) for s in itertools.product("+-", repeat=3)],
        help="repeat to verify selected regions; default verifies all eight",
    )
    parser.add_argument("--max-boxes", type=int, default=2_000_000)
    parser.add_argument("--report-every", type=int, default=100_000)
    parser.add_argument("--write-summary", type=Path)
    args = parser.parse_args()

    regions = args.signs or [
        "".join(signs) for signs in itertools.product("+-", repeat=3)
    ]
    summaries = [
        verify_region(signs, args.max_boxes, args.report_every)
        for signs in regions
    ]
    output = {
        "arithmetic": "fractions.Fraction",
        "target": f"{TARGET.numerator}/{TARGET.denominator}",
        "objective_pruning": "strictly below target",
        "local_radius": (
            f"{LOCAL_RADIUS.numerator}/{LOCAL_RADIUS.denominator}"
        ),
        "regions": summaries,
        "all_regions_closed": len(summaries) == 8,
    }
    rendered = json.dumps(output, indent=2)
    print(rendered)
    if args.write_summary:
        args.write_summary.write_text(rendered + "\n", encoding="utf-8")


if __name__ == "__main__":
    main()
\end{Verbatim}

\end{document}